\documentclass[reqno,11pt,a4paper]{article}

\usepackage{amsmath,amssymb,amsthm}
\usepackage[UKenglish]{babel}
\usepackage[right=2.54cm,left=2.54cm]{geometry}
\usepackage{ifthen}
\usepackage{multicol}
\usepackage{sectsty}
\usepackage[pdftitle={Stopping the CCR flow and its isometric cocycles},
            pdfauthor={Alexander C. R. Belton and Kalyan B. Sinha},
            bookmarks=true,
            colorlinks=false,
            pdfstartview=FitH]{hyperref}

\sectionfont{\large}
\subsectionfont{\normalsize}

\mathchardef\ordinarycolon\mathcode`\:
\mathcode`\:=\string"8000
\def\vcentcolon{\mathrel{\mathop\ordinarycolon}}
\begingroup \catcode`\:=\active
\lowercase{\endgroup
\let :\vcentcolon}

\theoremstyle{plain}
\newtheorem{theorem}{Theorem}[section]
\newtheorem{lemma}[theorem]{Lemma}
\newtheorem{proposition}[theorem]{Proposition}
\newtheorem{corollary}[theorem]{Corollary}

\theoremstyle{definition}
\newtheorem{definition}[theorem]{Definition}
\newtheorem{example}[theorem]{Example}
\newtheorem{notation}[theorem]{Notation}

\newtheorem{remark}[theorem]{Remark}

\newenvironment{mylist}%
{\begin{list}{}%
{\leftmargin 4.75em\labelwidth 3.5em\rightmargin 0.8em%
\topsep 0.75ex\itemsep 0.25ex}}%
{\end{list}}
{\begin{list}{}{}}{\end{list}}

\let\origthebibliography=\thebibliography
\def\thebibliography{\renewcommand{\section}[2]{}\origthebibliography}

\newcommand{\algten}{\mathbin{\odot}}
\newcommand{\bop}[2]%
{\ifthenelse{\equal{#2}{}}{\bopp( #1 )}{\bopp( #1; #2 )}}
\newcommand{\bopp}{B}
\newcommand{\borel}{\mathcal{B}}
\newcommand{\comp}{\mathbin{\circ}}

\newcommand{\elltwo}{L^2( \R_+; \mul )}
\newcommand{\evec}[1]{\evecc(#1)}
\newcommand{\evecc}{\varepsilon}
\newcommand{\evecs}{\mathcal{E}}
\newcommand{\fock}{\mathcal{F}}
\newcommand{\hlf}{\mbox{$\frac12$}}
\newcommand{\im}{\mathop{\mathrm{im}}}
\newcommand{\ini}{\mathsf{h}}
\newcommand{\intd}{\,\rd}

\newcommand{\mul}{\mathsf{k}}
\newcommand{\rd}{\mathrm{d}}
\newcommand{\stlim}{\mathop{\mathrm{st.lim}}}
\newcommand{\uwkten}{\mathbin{\overline{\otimes}}}
\newcommand{\vna}{\mathcal{M}}
\newcommand{\wh}[1]{\widehat{#1}}

\newcommand{\C}{\mathbb{C}}
\newcommand{\I}{\mathrm{i}}
\newcommand{\R}{\mathbb{R}}

\newcommand{\ie}{\textit{i.e., }}

\renewcommand{\ge}{\geqslant}
\renewcommand{\le}{\leqslant}

\numberwithin{equation}{section}

\begin{document}

\begin{center}
{\LARGE Stopping the CCR flow and its isometric cocycles}
\begin{multicols}{2}
{\large Alexander C.~R.~Belton}\\[0.5ex]
{\small Department of Mathematics and Statistics\\
Lancaster University, United Kingdom\\[0.5ex]
\textsf{a.belton@lancaster.ac.uk}}
\columnbreak

{\large Kalyan B.~Sinha}\\[0.5ex]
{\small Jawaharlal Nehru Centre for Advanced\\
Scientific Research, Bangalore, India\\[0.5ex]
\textsf{kbs@jncasr.ac.in}}
\end{multicols}
{\small 21st March 2013}
\end{center}

\begin{abstract}\noindent
It is shown how to use non-commutative stopping times in order to stop
the CCR flow of arbitrary index and also its isometric cocycles, \ie
left operator Markovian cocycles on Boson Fock space. Stopping the CCR
flow yields a homomorphism from the semigroup of stopping times,
equipped with the convolution product, into the semigroup of unital
endomorphisms of the von~Neumann algebra of bounded operators on the
ambient Fock space. The operators produced by stopping cocycles
themselves satisfy a cocycle relation.
\end{abstract}

{\footnotesize\textit{Key words:} quantum stopping time; quantum stop
time; quantum Markov time; operator cocycle; Markov cocycle; Markovian
cocycle; quantum stochastic cocycle; CCR flow.}

{\footnotesize\textit{MSC 2010:} %
46L53 (primary);   
46L55,             
60G40 (secondary). 
}

\section{Introduction}

Several authors have investigated the use of non-commutative stopping
times to stop quantum stochastic processes, beginning with the
pioneering work of Hudson \cite{Hud79}, Barnett and
Lyons~\cite{BaL86}, Parthasarathy and Sinha \cite{PaS87} and Sauvageot
\cite{Sau88}; in the framework of Fock-space quantum stochastic
calculus, more recent developments have been produced by Attal and
Sinha~\cite{AtS98}, Hudson~\cite{Hud07} and Coquio~\cite{Coq06}.

At its most general, a stopping time $S$ is an increasing,
time-indexed family $( S_t )_{t \in [ 0, \infty ]}$ of orthogonal
projections in a von~Neumann algebra which are subordinate to some
filtration. Below, the ambient von~Neumann algebra is $\bop{\fock}{}$,
the bounded operators on Boson Fock space over $\elltwo$, with $\mul$
an arbitrary complex Hilbert space; the filtration is that generated
by the increasing family of subspaces
$L^2\bigl( [ 0, t ]; \mul \bigr) \subseteq \elltwo$.

The CCR flow $\sigma$ is a semigroup $( \sigma_t )_{t \in \R_+}$ of
unital endomorphisms on $\bop{\fock}{}$ which arise from the isometric
right shift on $\elltwo$; CCR flows are fundamental examples of
$E_0$~semigroups~\cite{Arv03}. A stopped version $\sigma_S$ of the CCR
flow is constructed for any finite stopping time $S$, and the
composition of stopped flows is shown to be a homomorphism for the
stopping-time convolution introduced in \cite[Section~7]{PaS87}.

A cocycle $V$ for the CCR flow~$\sigma$ is a family
$( V_t )_{t \in \R_+}$ of bounded operators on $\ini \otimes \fock$,
where $\ini$ is an arbitrary complex Hilbert space, such that
\begin{equation}\label{eqn:cocycledef}
V_{s + t} = V_s \sigma_s( V_t ) \qquad \text{for all } s, t \in \R_+;
\end{equation}
the endomorphism $\sigma_s$ is extended to
$\bop{\ini \otimes \fock}{}$ by ampliation. (Although such cocycles
may be obtained by solving quantum stochastic differential equations,
we shall have no use for quantum stochastic calculus in this work.) It
is explained below how to stop such a cocycle $V$ with a stopping time
$S$ whenever the cocycle is isometric, strongly continuous and
satisfies a locality condition. Furthermore, the identity
\[
V_{S + t} = V_S \, \sigma_S( V_t )
\]
is shown to hold for all $t \in \R_+$; this is a non-deterministic
generalisation of the relation~\eqref{eqn:cocycledef}. These results
may be viewed as an extension of those obtained by Applebaum
\cite{App88}.

\section{Quantum stop times}

\begin{notation}
Let $\fock_A$ denote the Boson Fock space over $L^2( A; \mul )$, where
$A$ is a subinterval of~$\R_+$ and $\mul$ is a complex Hilbert
space which is fixed henceforth. Let $\fock := \fock_{\R_+}$,
$\fock_{t)} := \fock_{[ 0, t )}$ and
$\fock_{[t} := \fock_{[ t, \infty )}$ for all $t \in \R_+$, with
similar abbreviations for the identity operators $I$, $I_{t)}$
and~$I_{[t}$ on these spaces. Recall the tensor-product
decompositions
\[
\fock \cong \fock_{s)} \otimes \fock_{[s}; \ %
\evec{f} \mathbin{\leftrightarrow} %
\evec{f|_{[ 0, s )}} \otimes \evec{f|_{[ s, \infty )}}
\]
and
\[
\fock \cong \fock_{s)} \otimes \fock_{[s, t)} \otimes \fock_{[t}; \ %
\evec{f} \mathbin{\leftrightarrow} \evec{f|_{[ 0, s )}} \otimes %
\evec{f|_{[ s, t )}} \otimes \evec{f|_{[ t, \infty )}},
\]
where $s$, $t \in \R_+$ are such that $s < t$ and $\evec{g}$ is the
exponential vector corresponding to~$g$; these isomorphisms will be
used frequently without comment. Let~$\evecs$ denote the linear span
of the set of exponential vectors. For further details, see
\cite{Lin05} or \cite{Par92}.
\end{notation}

\begin{definition}\label{def:qst}
A \emph{quantum stop time} is a map
\[
S : \borel[ 0, \infty ] \to \bop{\fock}{},
\]
where $\borel[ 0, \infty ]$ is the Borel $\sigma$-algebra on the
extended half-line $[ 0, \infty ] := \R_+ \cup \{ \infty \}$, such
that
\begin{mylist}
\item[(i)] $S( A )$ is an orthogonal projection for all
$A \in \borel[ 0, \infty ]$;
\item[(ii)]
$\borel[ 0, \infty ] \to \C; \ %
A \mapsto \langle x, S( A ) y \rangle$
is a complex measure for all $x$,~$y \in \fock$;
\item[(iii)] $S\bigl( [ 0, \infty ] \bigr) = I$;
\item[(iv)] $S\bigl( \{ 0 \} \bigr) \in \C I$ and 
$S\bigl( [ 0, t ] \bigr) \in \bop{\fock_{t)}}{} \otimes I_{[t}$
for all $t \in ( 0, \infty )$.
\end{mylist}
\end{definition}
In other words, a quantum stop time is a spectral measure on
$\borel[ 0, \infty ]$ (conditions (i), (ii) and (iii)) which is
identity adapted (condition (iv)).

\begin{remark}
An equivalent way to view a quantum stop time is as an increasing,
identity-adapted family $( S_t )_{t \in [ 0, \infty ]}$ of orthogonal
projections on $\fock$ with $S_\infty = I$. Given a quantum stop time
in the sense of Definition~\ref{def:qst}, setting
$S_t : = S\bigl( [ 0, t ] \bigr)$ yields such a family; the spectral
theorem for self-adjoint operators may be used to move in the opposite
direction. Consequently, a quantum stop time $S$ may be defined by
specifying $S\bigl( [ 0, t ] \bigr)$ for all $t \in \R_+$.
\end{remark}

\begin{remark}
Conditions (i) and (iv) in Definition~\ref{def:qst} imply that
$S\bigl( \{ 0 \} \bigr)$ is either $0$ or $I$; as the latter
possibility leads to a trivial theory, which corresponds to stopping
immediately at~$0$, henceforth we shall require that any quantum stop
time $S$ is such that $S\bigl( \{ 0 \} \bigr) = 0$. The results below
take their most elegant form when $S$ is \emph{finite}, \ie
$S\bigl( \{ \infty \} \bigr) = 0$, so this condition is imposed as
well.
\end{remark}

\section{Time projections and right shifts}

\begin{definition}
Let $E_0 \in \bop{\fock}{}$ be the orthogonal projection onto
$\C \evec{0}$ and, for all $t \in ( 0, \infty )$, let
\[
E_t : \fock \to \fock; \ \evec{f} \mapsto \evec{1_{[ 0, t )} f}
\]
be the orthogonal projection onto
$\fock_{t)} \otimes \evec{0|_{[ t, \infty )}}$, where $1_{[ 0, t )}$
is the indicator function of the interval $[ 0, t )$. For all
$t \in \R_+$, let $\theta_t$ be the isometric right shift, so that
\[
\theta_t : \elltwo \to \elltwo; \ %
( \theta_t f )( s ) = 1_{[ t, \infty )}( s ) f( s - t ) = \left\{ %
\begin{array}{ll} f( s - t ) & \text{if } s \ge t, \\[1ex]
 0 & \text{if } 0 < s < t,
\end{array}\right.
\]
and let
\[
\Gamma_t : \fock \to \fock; \ \evec{f} \mapsto \evec{\theta_t f}
\]
be its second quantisation, which has adjoint
\[
\Gamma_t^* : \fock \to \fock; \ %
\evec{f} \mapsto \evec{\theta_t^* f},
\]
where $( \theta_t^* f )( s ) = f( s + t )$ for all $s \in \R_+$. As is
well known, $( \Gamma_t )_{t \in \R_+}$ and
$( \Gamma_t^* )_{t \in \R_+}$ are semigroups on $\fock$. Let
$E_\infty := I$, $\theta_\infty := 0$ and
$\Gamma_\infty := \Gamma( \theta_\infty ) = E_0$.

Note that $E_s \Gamma_t = E_0$ and $E_t \Gamma_s = \Gamma_s E_{t - s}$
for all $s$, $t \in \R_+ $ such that~$s \le t$. Furthermore,
the maps $t \mapsto \Gamma_t$ and $t \mapsto \Gamma_t^*$ are strongly
continuous on $[ 0, \infty )$, whereas $t \mapsto E_t$
and~$t \mapsto \Gamma_t \Gamma_t^*$ are strongly continuous on
$[ 0, \infty ]$.
\end{definition}

\begin{definition}
Let $J$ be a closed subinterval of $\R_+$ which contains $0$, so that
either $J = [ 0, t ]$ for some $t \in \R_+$ or $J = \R_+$. A
map $F : J \to \fock$ is \emph{future adapted} if and only if
\[
F( s ) = \evec{0|_{[ 0, s )}} \otimes F_s \in %
\evec{0|_{[ 0, s )}} \otimes \fock_{[s} %
\qquad \text{ for all } s \in J \setminus \{ 0 \}.
\]
If $J = \R_+$ then $F( \infty ) := \evec{0}$ and
$x \otimes F_\infty := x$ for all $x \in \fock$.
\end{definition}

\begin{notation}
For all $f$, $g \in \elltwo$, let $S^{f, g}$ be the finite Borel
measure on $[ 0, \infty ]$ such that
\[
S^{f, g}( A ) = \int_A \exp\Bigl( -\int_s^\infty %
\langle f( r ), g( r ) \rangle \intd r \Bigr) %
\langle \evec{f}, S( \rd s ) \evec{g} \rangle
\qquad \text{for all } A \in \borel[ 0, \infty ].
\]
\end{notation}

\begin{lemma}\label{lem:key}
Let $S$ be a finite quantum stop time, let $J$ be a closed subinterval
of $\R_+$ which contains $0$ and let $F : J \to \fock$ be
continuous, bounded and future adapted. For all $f \in \elltwo$, the
integral
\begin{equation}\label{eqn:mixint}
\int_{[ 0, t ]} S( \rd s ) \evec{f|_{[ 0, s )}} \otimes F_s = %
\lim_\pi \sum_{j = 1}^{n + 1} S\bigl( ( \pi_{j - 1}, \pi_j ] \bigr) %
\evec{f_{[ 0, \pi_j )}} \otimes F_{\pi_j}
\end{equation}
is well defined, where $t := \sup J$ and the limit is taken over the
collection of all finite partitions
$\pi = \{ 0 = \pi_0 < \cdots < \pi_{n + 1} = t \}$,
ordered by refinement. If $g \in \elltwo$ and $G : J \to \fock$ is
continuous, bounded and future adapted then
\begin{equation}\label{eqn:keyip}
\langle %
\int_{[ 0, t ]} S( \rd s ) \evec{f|_{[ 0, s )}} \otimes F_s, %
\int_{[ 0, t ]} S( \rd s ) \evec{g|_{[ 0, s )}} \otimes G_s %
\rangle = %
\int_{[ 0, t ]} \langle F( s ), G( s ) \rangle \, S^{f, g}( \rd s ),
\end{equation}
and if $r \in ( 0, \infty ]$ then
\begin{equation}\label{eqn:keyS}
S( [ 0, r ] ) %
\int_{[ 0, t ]} S( \rd s ) \evec{f|_{[ 0, s )}} \otimes F_s = %
\int_{[ 0, r \wedge t ]} %
S( \rd s ) \evec{f|_{[ 0, s )}} \otimes F_s,
\end{equation}
where $r \wedge t$ is the minimum of $r$ and $t$.
\end{lemma}
\begin{proof}
Let $I_{S, \pi}( f, F )$ denote the sum on the right-hand side of
\eqref{eqn:mixint}. Suppose $n \ge 1$ and let $\pi'$ be a refinement
of $\pi$; for $j = 0$, \ldots, $n$, let $k_j \ge j$ be such that
$\pi'_{k_j} = \pi_j$ and let~$l_j \ge 1$ be such
that~$\pi'_{k_j + l_j} = \pi_{j + 1}$. Then
\begin{align*}
&I_{S, \pi'}( f, F ) - I_{S, \pi}( f, F ) \\[1ex]
 & = \sum_{j = 0}^{n - 1} \sum_{l = 1}^{l_j} %
S\bigl( ( \pi'_{k_j + l - 1}, \pi'_{k_j + l} ] \bigr) %
\bigl( %
\evec{f|_{[ 0, \pi'_{k_j + l} )}} \otimes F_{\pi'_{k_j + l}} - %
\evec{f|_{[ 0, \pi'_{k_j + l_j} )}} \otimes F_{\pi'_{k_j + l_j}} %
\bigr) \\[1ex] 
 & \quad + \sum_{l = 1}^{l_n - 1} %
S\bigl( ( \pi'_{k_n + l - 1}, \pi'_{k_n + l} ] \bigr) %
\evec{f|_{[ 0, \pi'_{k_n + l} )}} \otimes F_{\pi'_{k_n + l}} - %
S\bigl( ( \pi_n, \pi'_{k_n + l_n - 1} ] \bigr) %
\evec{f|_{[ 0, t )}} \otimes F_t.
\end{align*}
If $r$, $s$, $u \in [ 0, t ]$ are such that $r < s \le u$ then
\[
\| S\bigl( ( r, s ] \bigr) %
\evec{f|_{[ 0, u )}} \otimes F_u \| \le %
\| S\bigl( ( r, s ] \bigr) \evec{f} \| \, \| F \|_\infty,
\]
where $\| F \|_\infty := \sup\{ \| F( x ) \| : x \in [ 0, t ] \}$; if,
also, $u < \infty$ then
\begin{align*}
\| S\bigl( ( r, s ] \bigr) %
\bigl( \evec{f&|_{[ 0, s )}} \otimes F_s - %
\evec{f|_{[ 0, u )}} \otimes F_u \bigr) \| \\[1ex]
 & \le \| S\bigl( ( r, s ] \bigr) \evec{1_{[ 0, s )} f} \| \, %
\bigl( \| F_s - \evec{0|_{[ s, u )}} \otimes F_u \| + %
\| \evec{0|_{[ s, u )}} - \evec{f|_{[ s, u )}} \| \, %
\| F_u \| \bigr) \\[1ex]
 & \le \| S\bigl( ( r, s ] \bigr) \evec{f} \| \, %
\bigl( \| F( s ) - F( u ) \| + %
\bigl( \| \evec{f|_{[ s, u )}} \|^2 - 1 \bigr)^{1 / 2} %
\| F \|_\infty \bigr).
\end{align*}
Consequently,
\begin{align*}
\| I_{S, \pi'}( f, F ) - I_{S, \pi}( f, F ) \| & \le %
\| S\bigl( [ 0, \pi_n ] \bigr) \evec{f} \| \bigl( %
\sup\{ \| F( r ) - F( \pi_j ) \| : %
r \in [ \pi_{j - 1}, \pi_j ], \ j = 1, \ldots, n \} \\
 & \hspace{7em} + \sup\{ \bigl( %
\| \evec{1_{[ \pi_{j - 1}, \pi_j )} f} \| - 1 \bigr)^{1 / 2} : %
j = 1, \ldots, n \} \, \| F \|_\infty \bigr) \\
 & \quad + 2 \| S\bigl( ( \pi_n, t ) \bigr) \evec{f} \| \, %
\| F \|_\infty
\end{align*}
and the integral exists. The inner-product identity is readily
verified.

For the final claim, note first that
$S\bigl( [ 0, r ] \bigr) I_{S, \pi}( f, F ) = I_{S, \pi}( f, F )$ if
$r \ge t$; otherwise, suppose without loss of generality that
$r = \pi_{m + 1}$ for some $m \ge 0$. Then
$S\bigl( [ 0, r ] \bigr) I_{S, \pi}( f, F ) = I_{S, \pi'}( f, F )$,
where $\pi' := \pi \cap [ 0, r ]$. As $\pi$ is refined, so is $\pi'$
and the result follows.
\end{proof}

\begin{corollary}\label{cor:key}
Let $S$ be a finite quantum stop time. If $J$ is a closed subinterval
of $\R_+$ which contains $0$ and has supremum $t$, and
$F : J \to \fock$ is Borel measurable, bounded and future adapted
then, for all $f \in \elltwo$, there exists a unique vector
$\int_{[ 0, t ]} S( \rd s ) \evec{f|_{[ 0, s )}} \otimes F_s %
\in \fock$
such that \eqref{eqn:keyip} and \eqref{eqn:keyS} hold for any
$g \in \elltwo$, any Borel-measurable, bounded and future-adapted
function $G : J \to \fock$ and any $r \in ( 0, \infty ]$.
\end{corollary}
\begin{proof}
For existence, note that $F$ may be approximated in a suitable sense
by a sequence of continuous functions: see
\cite[Propositions~4.9 and~4.10]{PaS87} for details. Uniqueness holds
because the exponential vector
$\evec{f}$ equals
$\int_{[ 0, \infty ]} S( \rd s ) \evec{f|_{[ 0, s )}} %
\otimes \Gamma_s \Gamma_s^* \evec{f}$,
for any $f \in \elltwo$.
\end{proof}

\begin{remark}
Lemma~\ref{lem:key} and Corollary~\ref{cor:key} give the existence of
a class of stop-time integrals which is smaller than that introduced
by Parthasarathy and Sinha \cite[Section~4]{PaS87} but sufficient for
present purposes.
\end{remark}

\begin{theorem}\label{thm:expS}
Let $S$ be a finite quantum stop time. For all $t \in ( 0, \infty ]$,
there exists an orthogonal projection $E_{S, t} \in \bop{\fock}{}$
such that
\[
E_{S, t} \evec{f} = \int_{[ 0, t ]} S( \rd s ) E_s \evec{f} := %
\int_{[ 0, t ]} S( \rd s ) \evec{f|_{[ 0, s )}} \otimes %
\evec{0|_{[ s, \infty )}} \qquad \text{for all } f \in \elltwo.
\]
Furthermore,
$E_{S, s} E_{S, t} = E_{S, s \wedge t} = E_{S, t} E_{S, s}$ and
$S\bigl( [ 0, s ] \bigr) E_{S, \infty} = E_{S, s}$, for all
$s \in ( 0, \infty ]$.
\end{theorem}
\begin{proof}
Let
$E_{S, \pi} := %
\sum_{j = 1}^{n + 1} S\bigl( ( \pi_{j - 1}, \pi_j ] \bigr) E_{\pi_j}$,
where $\pi = \{ 0 = \pi_0 < \cdots < \pi_{n + 1} = t \}$ is a
finite partition of $[ 0, t ]$. Note that, as $E_s$ commutes with
$S\bigl( ( r, s ] \bigr)$ whenever $r < s$, so $E_{S, \pi}$ is self
adjoint and idempotent, and therefore an orthogonal projection. Define
$E_{S, t}$ on $\evecs$ by using Lemma~\ref{lem:key} with
$F \equiv \evec{0}$ to let
\[
E_{S, t} \evec{f} := %
\int_{[ 0, t ]} S( \rd s ) \evec{f|_{[ 0, s )}} \otimes %
\evec{0|_{[ s, \infty )}} = %
\lim_\pi E_{S, \pi} \evec{f} \qquad \text{for all } f \in \elltwo
\]
and extending by linearity. Since each $E_{S, \pi}$ is a contraction,
so is $E_{S, t}$, hence $E_{S, t}$ extends by continuity to the whole
of $\fock$; self-adjointness and idempotency hold weakly on $\evecs$,
so everywhere.

For the second claim, suppose without loss of generality that
$s = \pi_{m + 1}$ for some $m \ge 1$ and
let~$\pi' = \{ 0 = \pi_0 < \cdots < \pi_{m + 1} = s \}$. Then
$E_{S, s} E_{S, t} = \lim_\pi E_{S, \pi'} E_{S, \pi} = %
\lim_\pi E_{S, \pi'} = E_{S, s}$
and $S\bigl( [ 0, s ] \bigr) E_{S, \infty} = %
\lim_\pi S\bigl( [ 0, s ] \bigr) E_{S, \pi} = %
\lim_\pi E_{S, \pi'} = E_{S, s}$,
in the weak and strong senses on $\evecs$ respectively, which gives
the result.
\end{proof}
 
\begin{theorem}\label{thm:shift}
Let $S$ be a finite quantum stop time. There exists an isometry
$\Gamma_S \in \bop{\fock}{}$ such that
\[
\Gamma_S x = %
\int_{[ 0, \infty ]} S( \rd s ) \Gamma_s x := %
\int_{[ 0, \infty ]} S( \rd s ) \evec{0|_{[ 0, s )}} \otimes %
\Gamma_s x \qquad \text{for all } x \in \fock,
\]
where $\Gamma_s$ is taken to have range
$\fock_{[s} \cong \evec{0|_{[ 0, s )}} \otimes \fock_{[s}$.
\end{theorem}
\begin{proof}
Let
$\Gamma_{S, \pi} := \sum_{j = 1}^{n + 1} %
S\bigl( ( \pi_{j - 1}, \pi_j ] \bigr) \Gamma_{\pi_j}$,
where $\pi = \{ 0 = \pi_0 < \cdots < \pi_{n + 1} = \infty \}$ is a
finite partition of $[ 0, \infty ]$. For all $x$, $y \in \fock$,
\[
\langle \Gamma_{S, \pi} x, \Gamma_{S, \pi} x \rangle = %
\sum_{j = 1}^{n + 1} \langle \Gamma_{\pi_j} x, %
S\bigl( ( \pi_{j - 1}, \pi_j ] \bigr) \Gamma_{\pi_j} y \rangle = %
\sum_{j = 1}^{n + 1} %
\| S\bigl( ( \pi_{j - 1}, \pi_j ] \bigr) \evec{0} \|^2 %
\langle \Gamma_{\pi_j} x, \Gamma_{\pi_j} y \rangle = %
\langle x, y \rangle,
\]
so $\Gamma_{S, \pi}$ is an isometry. Since $\Gamma_{S, \pi} x$
converges to
$\int_{[ 0, \infty ]} S( \rd s ) \evec{0|_{[ 0, s )}} \otimes %
\Gamma_s x$
as $\pi$ is refined, applying Lemma~\ref{lem:key} with~$J = \R_+$,
$f = 0$ and $F = \Gamma x : s \mapsto \Gamma_s x$ gives the claim.
\end{proof}

\begin{definition}
Let $S$ be a finite quantum stop time, let $E_S := E_{S, \infty}$ and
let $\im T$ denote the range of the linear operator $T$. The
\emph{pre $S$ space}
\[
\fock_{S)} := \im E_S = \bigcup_{t \in ( 0, \infty ]} \im E_{S, t},
\]
and the \emph{post $S$ space} $\fock_{[S} := \im \Gamma_S$, with
identity operators $I_{S)}$ and $I_{[S}$ respectively. If $S$ is
deterministic, \ie $S\bigl( \{ t \} \bigr) = I$ for some
$t \in ( 0, \infty )$, then
$\fock_{[S} = \evec{0|_{[ 0, t )}} \otimes \fock_{[t}$, which, as
noted above, may
naturally be identified with $\fock_{[t}$.
\end{definition}

\begin{theorem}\label{thm:isom}
Let $S$ be a finite quantum stop time. There exists an isometric
isomorphism $j_S : \fock_{S)} \otimes \fock_{[S} \to \fock $ such that
\[
j_S\bigl( E_{S, t} \evec{f} \otimes \Gamma_S x \bigr) = %
\int_{[ 0, t ]} S( \rd s ) \evec{f|_{[ 0, s )}} \otimes \Gamma_s x %
\quad \text{for all } t \in ( 0, \infty ], \ f \in \elltwo %
\text{ and } x \in \fock.
\]
\end{theorem}
\begin{proof}
Let $t \in ( 0, \infty ]$, $f$, $g \in \elltwo$ and
$x$, $y \in \fock$. Lemma~\ref{lem:key} and Theorem~\ref{thm:expS}
imply that
\begin{align*}
\langle \int_{[ 0, t ]} S( \rd s ) \evec{1_{[ 0, s )} f} \otimes %
\Gamma_s x, \int_{[ 0, t ]} S( \rd s ) \evec{1_{[ 0, s )} g} %
\otimes \Gamma_s y \rangle & = %
\int_{[ 0, t ]} \langle \Gamma_s x, \Gamma_s y \rangle %
\, S^{f, g}( \rd s ) \\[1ex]
 & = \langle E_{S, t} \evec{f}, E_{S, t} \evec{f'} \rangle %
\langle x, y \rangle.
\end{align*}
Hence $j_S$ is well defined and extends to an isometry
from~$\fock_{S)} \otimes \fock_{[S}$ to $\fock$.

To conclude, it suffices to show that $\evec{f} \in \im j_S$ for any
$f \in \elltwo$. Let $\pi'$ be a refinement of the partition
$\pi = \{ 0 = \pi_0 < \cdots < \pi_{n + 1} = \infty \}$ and, for
$j = 0$, \ldots, $n$, let $k_j \ge j$ be such
that~$\pi'_{k_j} = \pi_j$ and let $l_j \ge 1$ be such that
$\pi'_{k_j + l_j} = \pi_{j + 1}$. Then
\[
\evec{f} = %
\sum_{j = 0}^n \sum_{l = 1}^{l_j} %
S\bigl( ( \pi'_{k_j + l - 1}, \pi'_{k_j + l} ] \bigr) %
\evec{1_{[ 0, \pi'_{k_j + l} )} f + \theta_{\pi'_{k_j + l}} %
\theta^*_{\pi'_{k_j + l}} f}
\]
and, as $\pi'$ is refined with $\pi$ fixed,
\begin{multline*}
\sum_{j = 0}^n \sum_{l = 1}^{l_j} %
S\bigl( ( \pi'_{k_j + l - 1}, \pi'_{k_j + l} ] \bigr) %
\bigl( \evec{1_{[ 0, \pi'_{k_j + l} )} f + \theta_{\pi'_{k_j + l}} %
\theta^*_{\pi'_{k_j + l}} f} - %
\evec{1_{[ 0, \pi'_{k_j + l} )} f + \theta_{\pi'_{k_j + l}} %
\theta^*_{\pi'_{k_j + l_j}} f} \bigr) \\
\to R_\pi := \evec{f} - %
\sum_{j = 0}^n j_S\bigl( ( E_{S, \pi_{j + 1}} - %
E_{S, \pi_j} ) \evec{f} \otimes \Gamma_S %
\evec{\theta_{\pi_{j + 1}}^* f} \bigr);
\end{multline*}
note that $s \mapsto \Gamma_s \evec{\theta_{\pi_{j + 1}}^* f}$ is
continuous on $[ 0, \pi_{j + 1} ]$ for $j = 0$, \ldots, $n$. If $r$,
$s$, $u \in \R_+$ are such that $r < s \le u$ then
\begin{multline*}
\| S\bigl( ( r, s ] \bigr) \bigl( %
\evec{1_{[ 0, s )} f + \theta_s \theta^*_s f} - %
\evec{1_{[ 0, s )} f + \theta_s \theta^*_u f} \bigr) \| \\[1ex]
 \le \| S\bigl( ( r, s ] \bigr) \evec{f} \| \, %
\| \evec{1_{[ s, \infty )} f} - %
\evecc\bigl( 1_{[ s, \infty )} f( \cdot + u - s ) \bigr) \|
\end{multline*}
and, for all $h \ge 0$,
\[
\| \evec{1_{[ s, \infty )} f} - %
\evecc\bigl( 1_{[ s, \infty )} f( \cdot + h ) \bigr) \|^2 \le %
\eta( h ) := 2 \exp\bigl( \| f \|^2 \bigr) %
\bigl( \exp\bigl( \| f \| \, %
\| f - f( \cdot + h ) \| \bigr) - 1 \bigr).
\]
Letting
$\delta_\pi := \max\{ \pi_j - \pi_{j - 1} : j = 1, \ldots, n \}$ and
refining $\pi$, it follows that
\[
\| R_\pi \| \le \| S\bigl( [ 0, \pi_n ] \bigr) \evec{f} \| %
\sup\{ \eta( h )^{1 / 2} : 0 \le h \le \delta_\pi \} + %
2 \| S\bigl( ( \pi_n, \infty ) \bigr) \evec{f} \| \to 0.
\qedhere
\]
\end{proof}

\begin{remark}
Use of the time projection $E_S$ goes back at least as far as
\cite[Theorem~2.3]{BaT90}, and the existence of the stopped right
shift $\Gamma_S$ was first established in \cite[Theorem~5.1]{PaS87}.
The abstract strong Markov property, Theorem~\ref{thm:isom}, is
\cite[Theorem~5.2]{PaS87}.
\end{remark}

\section{Stop-time convolution}

\begin{definition}
Let $S$ and $T$ be finite quantum stop times. There exists a unique
spectral measure $S \otimes T$ on $\borel( [ 0, \infty ]^2 )$ such
that
\[
( S \otimes T )( A \times B ) = %
j_S( S( A ) \otimes \Gamma_S T( B ) \Gamma_S^* ) j_S^* %
\qquad \text{for all } A, B \in \borel[ 0, \infty ],
\]
where $S( A )$ is considered to act on $\fock_{S)}$, so setting
\[
( S \star T )( C ) := ( S \otimes T )%
\bigl( \{ ( s, t ) \in [ 0, \infty ] : s + t \in C \} \bigr) %
\qquad \text{for all } C \in \borel[ 0, \infty ]
\]
gives a finite spectral measure $S \star T$, their \emph{convolution}.

The convolution of spectral measures was introduced by Fox
\cite{Fox76} and first applied to quantum stop times
in~\cite[Proposition~7.1]{PaS87}.
\end{definition}

\begin{theorem}\label{thm:SstarT}
The convolution of finite quantum stop times $S$ and $T$ is such that
\begin{equation}\label{eqn:SstarTint}
( S \star T )\bigl( [ 0, u ] \bigr) \evec{f} = %
\int_{[ 0, u ]} S( \rd s ) \evec{f|1_{[ 0, s )} } \otimes %
\Gamma_s T\bigl( [ 0, u - s ] \bigr) \Gamma_s^* \evec{f}
\end{equation}
for all $u \in [ 0, \infty ]$ and $f \in \elltwo$. Consequently,
$S \star T$ is a finite quantum stop time such that
\[
E_{S \star T} \evec{f} = %
\int_{[ 0, \infty ]} S( \rd s ) \evec{f|_{[ 0, s )}} \otimes %
\Gamma_{s} E_T \Gamma_s^* \evec{f}
\]
and
\begin{equation}\label{eqn:jSstarT}
j_{S \star T}%
( E_{S \star T} \evec{f} \otimes \Gamma_{S \star T} y ) = %
\int_{[ 0, \infty ]} S( \rd s ) \evec{f|_{[ 0, s )}} \otimes %
\Gamma_s %
\int_{[ 0, \infty ]} T( \rd t ) \evec{f( \cdot + s )|_{[ 0, t )}} %
\otimes \Gamma_t y
\end{equation}
for all $f \in \elltwo$ and $y \in \fock$.
\end{theorem}

\begin{remark}
For any $f \in \elltwo$ and $y \in \fock$, the integral
$\int_{[ 0, \infty ]} T( \rd t ) \evec{f( \cdot + s )|_{[ 0, t )}} %
\otimes \Gamma_t y$ is the limit of a sequence of Riemann sums which
depend continuously on $s$. Furthermore, the norm of these integrals
is uniformly bounded above, by $\| \evec{f} \otimes y \|$, so the
iterated integral in~\eqref{eqn:jSstarT} is well defined.
\end{remark}

\begin{proof}[Proof of Theorem~\ref{thm:SstarT}]
Let $f$, $g \in \elltwo$ and let $u \in ( 0, \infty )$; the cases with
$u = 0$ and $u = \infty$ are clear. The proof of
Theorem~\ref{thm:isom} gives that
\[
\evec{f} = \lim_\pi %
\sum_{j = 1}^{n + 1} j_S\bigl( ( E_{S, \pi_j} - E_{S, \pi_{j - 1}} ) %
\evec{f} \otimes \Gamma_S \evec{\theta_{\pi_j}^* f} \bigr),
\]
where the limit is taken over finite partitions
$\pi = \{ 0 = \pi_0 < \cdots < \pi_{n + 1} = \infty \}$ ordered by
refinement; without loss of generality, $u = \pi_{m + 1}$ for some $m$
which depends on $\pi$. Furthermore,
\[
E_{S, t} - E_{S, s} = E_S S\bigl( ( s, t ] \bigr) = %
S\bigl( ( s, t ] \bigr) E_S %
\qquad \text{whenever } 0 \le s < t \le \infty
\]
It follows that
$\langle \evec{f}, %
( S \star T )\bigl( [ 0, u ] \bigr) \evec{g} \rangle$ is the limit, as
$\pi$ as refined, of
\begin{align*}
& \sum_{j, k = 1}^{n + 1} \iint\limits_{\quad 0 \le s + t \le u} \! %
\langle ( E_{S, \pi_j} - E_{S, \pi_{j - 1}} ) \evec{f} \otimes %
\evec{\theta_{\pi_j}^* f}, %
\bigl( S( \rd s ) \otimes T( \rd t ) \bigr)
( E_{S, \pi_k} - E_{S, \pi_{k - 1}} ) \evec{g} \otimes %
\evec{\theta_{\pi_k}^* g} \rangle \\[1ex]
 & = \sum_{j, k = 1}^{n + 1} %
\int_{[ 0, u ]} \langle S\bigl( ( \pi_{j - 1}, \pi_j ] \bigr) E_S %
\evec{f}, S\bigl( [ 0, u - t ] \bigr) %
S\bigl( ( \pi_{k - 1}, \pi_k ] \bigr) E_S \evec{g} \rangle %
\langle \Gamma_{\pi_j}^* \evec{f}, T( \rd t ) %
\Gamma_{\pi_k}^* \evec{g} \rangle \\[1ex]
 & = \sum_{j = 1}^{m + 1} \int_{( \pi_{j - 1}, \pi_j ]} %
\langle \evec{f}, S( \rd s ) %
S\bigl( ( \pi_{j - 1}, \pi_j ] \bigr) E_S \evec{g} \rangle %
\langle \evec{f}, \Gamma_{\pi_j} T\bigl( [ 0, u - s ] \bigr) %
\Gamma_{\pi_j}^* \evec{g} \rangle \\[1ex]
 & = \sum_{j = 1}^{m + 1} \int_{( \pi_{j - 1}, \pi_j ]} %
\langle \evec{f}, S( \rd s ) E_S \evec{g} \rangle %
\langle \evec{f}, \Gamma_{\pi_j} T\bigl( [ 0, u - s ] \bigr) %
\Gamma_{\pi_j}^* \evec{g} \rangle \\[1ex]
 & \to \int_{[ 0, u ]} %
\langle \evec{f}, S( \rd s ) E_S \evec{g} \rangle %
\langle \evec{f}, %
\Gamma_s T\bigl( [ 0, u - s ] \bigr) \Gamma_s^* \evec{g} \rangle.
\end{align*}
If $F : [ 0, u ] \to \C$ is continuous then
\[
\int_{[ 0, u ]} %
\langle \evec{f}, S( \rd s ) E_S \evec{g} \rangle F( s ) = %
\lim_\pi \sum_{j = 1}^{m + 1} \langle \evec{f}, %
S\bigl( ( \pi_{j - 1}, \pi_j ] \bigr) E_{\pi_j} \evec{g} \rangle %
F( \pi_j ) = %
\int_{[ 0, u ]} S^{f, g}( \rd s ) F( s ),
\]
so \eqref{eqn:SstarTint} now follows from Lemma~\ref{lem:key} and the
proof of Corollary~\ref{cor:key}. Furthermore, as $S$ and $T$ are
identity adapted, so is $S \star T$: note that if $s$, $u \in \R_+$
are such that $s \le u$ then
\[
\langle \evec{f}, %
\Gamma_s T\bigl( [ 0, u - s ] \bigr) \Gamma_s^* \evec{g} \rangle = %
\langle \evec{1_{[ 0, u )} f}, \Gamma_s T\bigl( [ 0, u - s ] \bigr) %
\Gamma_s^* \evec{1_{[ 0, u )} g} \rangle \langle %
\evec{1_{[ u, \infty )} f}, \evec{1_{[ u, \infty )} g} \rangle.
\]
Next, note that
\begin{align*}
\langle \evec{f}, E_{S \star T, \pi} \evec{g} \rangle & = %
\sum_{j = 1}^{n + 1} \int_{[ 0, \pi_j ]} \langle \evec{f}, \Gamma_s %
T\bigl( [ ( \pi_{j - 1} - s )_+, \pi_j - s ] \bigr) E_{\pi_j - s} %
\Gamma_s^* \evec{g} \rangle \, S^{f, g}( \rd s ) \\[1ex]
 & = \int_{[ 0, \infty ]} \langle \evec{f}, \Gamma_s E_{T, \pi - s} %
\Gamma_s^* \evec{g} \rangle \, S^{f, g}( \rd s ),
\end{align*}
where $x_+ := \max\{ x, 0 \}$ for all $x \in \R$ and
$\pi - s$ is the finite partition of $[ 0, \infty ]$ consisting of
intervals with end points
$\{ ( \pi_0 - s )_+ , \cdots, ( \pi_n - s )_+, \infty \}$; the
integral form of $E_{S \star T} \evec{f}$ now follows after refining
$\pi$. For the last identity, note that
\begin{align*}
\langle \evec{f}&, j_{S \star T}( E_{S \star T} \evec{g} \otimes %
\Gamma_{S \star T} y ) \rangle \\[1ex]
 & = \lim_\pi \sum_{j = 1}^{n + 1} \int_{[ 0, \pi_j ]} \langle %
\Gamma_s \Gamma_s^* \evec{f}, %
\Gamma_s T\bigl( [ ( \pi _{j - 1} - s)_+, \pi_j - s ] \bigr) %
\Gamma_s^* \evec{g} \rangle \langle \evec{f}, %
\Gamma_{\pi_j} y \rangle S^{f, g}( \rd s ) \\[1ex]
 & = \lim_\pi \int_{[ 0, \infty ]} \langle \evec{f}, \Gamma_s\Bigl( %
\smash[b]{\sum_{j = 1}^{n + 1}} 1_{[ 0, \pi_j ]}( s ) %
T\bigl( [ ( \pi _{j - 1} - s)_+, \pi_j - s ] \bigr) %
\evec{g( \cdot + s )|_{[ 0, \pi_j - s )}} \otimes %
\Gamma_{\pi_j - s} y \Bigr) \rangle \\
 & \hspace{18em} \times %
\smash[b]{\exp\Bigl( -\int_{\pi_j - s}^\infty %
\langle f( r + s ), g( r + s ) \rangle \intd r \Bigr)} %
S^{f, g}( \rd s ) \\[1ex] 
 & = \lim_\pi \int_{[ 0, \infty ]} \langle \evec{f}, \Gamma_s %
I_{T, \pi - s}( g( \cdot + s ), \Gamma y ) \rangle S^{f, g}( \rd s ) %
\\[1ex]
 & = \langle \evec{f}, \int_{[ 0, \infty ]} S( \rd s ) %
\evec{g|_{[ 0, s )}} \otimes \Gamma_s \int_{[ 0, \infty ]} %
T( \rd t ) \evec{g( \cdot + s )|_{[ 0, t )}} \otimes \Gamma_t y %
\rangle,
\end{align*}
as required, where $I_{T, \pi - s}$ is as in the proof of
Lemma~\ref{lem:key} and $(\Gamma y )( t ) := \Gamma_t y$ for all
$t \in [ 0, \infty ]$.
\end{proof}

\begin{remark}
The identity \eqref{eqn:SstarTint} may also be used to show that
\[
\Gamma_{S \star T} x = \int_{[ 0, \infty ]} S( \rd s ) %
\evec{0|_{[ 0, s )}} \otimes \Gamma_s E_T x = %
\Gamma_S \bigl( \Gamma_T x \bigr) \qquad \text{for all } x \in \fock
\]
and for any finite quantum stop times $S$ and $T$,
so that $\Gamma_{S \star T} = \Gamma_S \comp \Gamma_T$. As we shall
have no use for this result \cite[Proposition~7.6]{PaS87}, its proof
is left as an exercise.
\end{remark}

\begin{proposition}
Let $S$ and $T$ be finite quantum stop times and suppose $T$ is
deterministic, \ie $T\bigl( \{ t \} \bigr) = I$ for
some~$t \in ( 0, \infty )$. Then $S \star T = S + t$, where
\[
( S + t )\bigl( [ 0, u ] \bigr) = \left\{ \begin{array}{cl}
 0 & \text{if } u \in [ 0, t ), \\[1ex]
 S\bigl( [ 0, u - t ] \bigr) & \text{if } u \in [ t, \infty ].
\end{array}\right.
\]
\end{proposition}
\begin{proof}
If $u < t$ then $T\bigl( [ 0, u - s ] \bigr) = 0$ for all
$s \in [ 0, u ]$ and therefore
\[
( S \star T )\bigl( [ 0, u ] \bigr) \evec{f} = %
\int_{[ 0, u ]} S( \rd s ) \evec{f|_{[ 0, s ]}} \otimes %
\Gamma_s T\bigl( [ 0, u - s ] \bigr) \Gamma_s^* \evec{f} = 0 %
\quad \text{for all } f \in \elltwo.
\]
If $u \ge t$ then $T\bigl( [ 0, u - s ] \bigr)$ equals $I$ for all
$s \in [ 0, u - t ]$ and equals $0$ for all $s \in ( u - t, u ]$, so
\[
( S \star T )\bigl( [ 0, u ] \bigr) \evec{f} = %
\int_{[ 0, u - t ]} S( \rd s ) \evec{f|_{[ 0, s ]}} \otimes %
\Gamma_s \Gamma_s^* \evec{f} = S\bigl( [ 0, u - t ] \bigr) \evec{f} %
\quad \text{for all } f \in \elltwo,
\]
where the second equality holds by \eqref{eqn:keyS}.
\end{proof}

\begin{theorem}\label{thm:SstarTiso}%
{\textup{(\textit{Cf}.\ \cite[Theorem~7.5]{PaS87})}}
Let $S$ and $T$ be finite quantum stop times. Then
\[
j_{S, T} : \fock_{S)} \otimes \Gamma_S( \fock_{T)} ) \otimes %
\fock_{[S \star T} \to \fock; \ %
E_S x \otimes \Gamma_S E_T y \otimes \Gamma_{S \star T} z \mapsto %
j_S\bigl( %
E_S x \otimes \Gamma_S j_T( E_T y \otimes \Gamma_T z ) \bigr)
\]
is an isometric isomorphism such that
\begin{align}
j_{S, T}\bigl( \fock_{S)} \otimes \Gamma_S( \fock_{T)} ) \otimes %
\evec{0} \bigr) & = %
j_S\bigl( \fock_{S)} \otimes \Gamma_S( \fock_{T)} ) \bigr) = %
\fock_{S \star T)} \label{eqn:SstarT1} \\[1ex]
\text{and} \qquad %
j_{S, T}\bigl( \evec{0} \otimes \Gamma_S( \fock_{T)} ) \otimes %
\fock_{[S \star T} \bigr) & = \fock_{[S}. \label{eqn:SstarT2}
\end{align}
Furthermore,
\[
j_{S, T} = j_{S \star T} \comp ( j_S\bigr|_{F_{S)} \otimes %
\Gamma_S( \fock_{T)} )} \otimes I_{[S \star T} ).
\]
\end{theorem}
\begin{proof}
Note first that
\[
j_{S, T} = j_S \comp ( I_{S)} \otimes \Gamma_S ) \comp %
( I_{S)} \otimes j_T ) \comp %
( I_{S)} \otimes \Gamma_S^* \otimes \Gamma_T \Gamma_{S \star T}^* );
\]
as each of the maps on the right-hand side is an isometric isomorphism
between the appropriate spaces, so is $j_{S, T}$. Next, let
$f$, $g \in \elltwo$ and $y \in \fock$; it follows from
Theorem~\ref{thm:SstarT} that
\begin{align*}
\langle E_{S \star T} \evec{f}, %
j_S( E_S \evec{g} \otimes \Gamma_S y ) \rangle & = %
\int_{[ 0, \infty ]} \langle \Gamma_s E_T \Gamma_s^* \evec{f}, %
\Gamma_s y \rangle S^{f, g}( \rd s ) \\[1ex]
 & = \int_{[ 0, \infty ]} \langle \Gamma_s \Gamma_s^* \evec{f}, %
\Gamma_s E_T y \rangle S^{f, g}( \rd s ) \\[1ex]
 & = \langle \evec{f}, %
j_S( E_S \evec{g} \otimes \Gamma_S E_T y ) \rangle.
\end{align*}
Thus
\[
E_{S \star T} j_S( E_S x \otimes \Gamma_S y ) = %
j_S( E_S x \otimes \Gamma_S E_T y ) = %
j_{S, T}\bigl( E_S x \otimes \Gamma_S E_T y \otimes \evec{0} \bigr) %
\quad \text{for all } x, y \in \fock;
\]
as $j_S$ is surjective, \eqref{eqn:SstarT1} follows. That
\eqref{eqn:SstarT2} holds is an immediate consequence of the identity
\[
j_{S, T}%
( \evec{0} \otimes \Gamma_S E_T y \otimes \Gamma_{S \star T} z ) = %
\Gamma_S j_T( E_T y \otimes \Gamma_T z ) %
\qquad \text{for all } y, z \in \fock
\]
and the surjectivity of $j_T$.

Finally, if $f$, $g$, $h \in \elltwo$ and $y$, $z \in \fock$ then
\begin{align*}
j_{S, T}( E_S \evec{f} \otimes \Gamma_S E_T \evec{g} \otimes %
\Gamma_{S \star T} y ) & = j_S\bigl( %
E_S \evec{f} \otimes \Gamma_S %
j_T( E_T \evec{g} \otimes \Gamma_T y ) \bigr) \\[1ex]
 & = \int_{[ 0, \infty ]} S( \rd s ) \evec{f|_{[ 0, s )}} \otimes %
\Gamma_s \int_{[ 0 , \infty ]} T( \rd t ) \evec{g|_{[ 0, t )}} %
\otimes \Gamma_t y,
\end{align*}
so Theorem~\ref{thm:SstarT} implies that
\begin{align*}
\langle j_{S, T}( E_S \evec{f&} \otimes \Gamma_S E_T \evec{g} \otimes %
\Gamma_{S \star T} y ), j_{S \star T}( E_{S \star T} \evec{h} %
\otimes \Gamma_{S \star T} z ) \rangle \\[1ex]
 & = \int_{[ 0, \infty ]} \bigl\langle \int_{[ 0, \infty ]} %
T( \rd t ) \evec{g|_{[ 0, t )}} \otimes \Gamma_t y, %
\int_{[ 0, \infty ]} T( \rd t ) \evec{h( \cdot + s )|_{[ 0, t )}} %
\otimes \Gamma_t z \bigr\rangle \, S^{f, h}( \rd s ) \\[1ex]
 & = \int_{[ 0, \infty ]} \int_{[ 0, \infty ]} \langle y, z \rangle %
\, T^{g, h( \cdot + s )}( \rd t ) \, S^{f, h}( \rd s ) \\[1ex]
 & = \int_{[ 0, \infty ]} \langle E_T \evec{g}, E_T \Gamma_s^* %
\evec{h} \rangle \, S^{f, h}( \rd s ) \langle y, z \rangle \\[1ex]
 & = \langle j_S( E_S \evec{f} \otimes \Gamma_S E_T \evec{g} ), %
\evec{h} \rangle \langle y, z \rangle \\[1ex]
 & = \langle j_{S \star T}( j_S\bigl( E_S \evec{f} \otimes %
\Gamma_S E_T \evec{g} \bigr) \otimes \Gamma_{S \star T} y ), %
j_{S \star T}( E_{S \star T} \evec{h} \otimes \Gamma_{S \star T} z ) %
\rangle.
\qedhere
\end{align*}
\end{proof}

\section{The CCR flow}

\begin{definition}\label{def:ccrflow}
For all $t \in \R_+$, define a unital $*$-homomorphism
\[
\sigma_t : \bop{\fock}{} \to \bop{\fock}{}; \ %
X \mapsto I_{t)} \otimes \Gamma_t X \Gamma_t^*,
\]
where $\fock_{[t}$ is regarded as an isometric isomorphism from
$\fock$ onto~$\fock_{[t}$ with inverse $\Gamma_t^*$. Note
that~$\sigma_t$ has range
$\im \sigma_t = I_{t)} \otimes \bop{\fock_{[t}}{}$, that
$\sigma_s( E_t ) = E_{s + t}$ for all $s \in \R_+$ and that
\begin{equation}\label{eqn:shift}
\sigma_t( X ) \Gamma_t = \Gamma_t X %
\qquad \text{for all } X \in \bop{\fock}{}.
\end{equation}
Let $\sigma_{\infty}$ be the constant map on $\bop{\fock}{}$ with
value $I$.

As is well known, $( \sigma_t )_{t \in \R_+}$ is a semigroup on
$\bop{\fock}{}$ called the \emph{CCR flow}.
\end{definition}

\begin{theorem}\label{thm:flowstop}
Let $S$ be a finite quantum stop time and let
$\pi = \{ 0 = \pi_0 < \cdots < \pi_{n + 1} = \infty \}$ be a finite
partition of $[ 0, \infty ]$. The map
\[
\sigma_{S, \pi} : \bop{\fock}{} \to \bop{\fock}{}; \ %
X \mapsto \sum_{j = 1}^{n + 1} %
\sigma_{\pi_j}( X ) S\bigl( ( \pi_{j - 1}, \pi_j ] \bigr)
\]
is a unital $*$-homomorphism, the limit
$\sigma_S( X ) := \stlim_\pi \sigma_{S, \pi}( X )$
exists for all $X \in \bop{\fock}{}$ and the map
\[
\sigma_S : \bop{\fock}{} \to \bop{\fock}{}; \ %
X \mapsto \sigma_S( X )
\]
is a unital $*$-homomorphism such that
\begin{equation}\label{eqn:flowgamma}
\sigma_S( X ) \Gamma_S = \Gamma_S X %
\qquad \text{for all } X \in \bop{\fock}{}.
\end{equation}
\end{theorem}
\begin{proof}
Since
\[
\sigma_s( X ) \in I_{s)} \otimes \bop{\fock_{[s}}{} %
\qquad \text{and} \qquad %
S\bigl( ( r, s ] \bigr) \in \bop{\fock_{s)}}{} \otimes I_{[s}
\]
for all $r$, $s \in \R_+$ with $r < s$, the first claim is immediate.
For the next, suppose $n \ge 1$, let~$\pi'$ be a refinement of $\pi$
and, for $j = 0$, \ldots, $n$, let $k_j \ge j$ be such that
$\pi'_{k_j} = \pi_j$ and let~$l_j \ge 1$ be such
that~$\pi'_{k_j + l_j} = \pi_{j + 1}$. Then
\begin{align*}
\bigl( \sigma_{S, \pi'} - \sigma_{S, \pi} \bigr)( X ) \evec{f} & = %
 \sum_{j = 0}^{n - 1} \sum_{l = 1}^{l_j} %
S\bigl( ( \pi'_{k_j + l - 1}, \pi'_{k_j + l} ] \bigr) %
( \sigma_{\pi'_{k_j + l}} - \sigma_{\pi'_{k_j + l_j}} )( X ) %
\evec{f} \\
 & \quad + \sum_{l = 1}^{l_n - 1} %
S\bigl( ( \pi'_{k_n + l - 1}, \pi'_{k_n + l} ] \bigr) %
\sigma_{\pi'_{k_n + l}}( X ) \evec{f} - %
S\bigl( ( \pi_n, \pi'_{k_n + l_n - 1} ] \bigr) \evec{f}.
\end{align*}
If $r$, $s$, $t \in \R_+$ are such that $r < s \le t$ then
\begin{align*}
\| S\bigl( ( r, s ] \bigr) %
\bigl( \sigma_t( X ) - \sigma_s( X ) \bigr) \evec{f} \| & = %
\| S\bigl( ( r, s ] \bigr) \evec{1_{[ 0 , s )} f} \| \, %
\| \bigl( \sigma_t( X ) - \sigma_s( X ) \bigr) %
\evec{1_{[ s, \infty )} f} \| \\[1ex]
 & \le \| S\bigl( ( r, s ] \bigr) \evec{f} \| \, %
\| \sigma_s( \sigma_{t - s}( X ) - X ) \Gamma_s %
\Gamma_s^* \evec{f} \| \\[1ex]
 & = \| S\bigl( ( r, s ] \bigr) \evec{f} \| \, %
\| ( \sigma_{t - s}( X ) - X ) %
\evecc\bigl( f( \cdot + s ) \bigr) \|,
\end{align*}
by \eqref{eqn:shift}; furthermore,
\[
\| S\bigl( ( r, s ] \bigr) \sigma_s( X ) \evec{f} \| \le %
\| S\bigl( ( r, s ] \bigr) \evec{f} \| \, \| X \|.
\]
Hence if $f$ has support contained in
$[ 0, T ] \subseteq [ 0, \pi_m ]$ and
$\delta_\pi := \max\{ \pi_j - \pi_{j - 1} : j = 1, \ldots, m \}$ then
\begin{align*}
\| \bigl( \sigma_{S, \pi'} - \sigma_{S, \pi} \bigr)( X ) & \evec{f} \|
\\[1ex]
 & \le \| S\bigl( [ 0, \pi_n ] \bigr) \evec{f} \| %
\sup\{ \| ( \sigma_r( X ) - X ) %
\evecc\bigl( f( \cdot + s ) \bigr) \| : %
r \in [ 0, \delta_\pi ], \ s \in [ 0, T ] \} \\[1ex]
 & \qquad + %
\| S\bigl( ( \pi_n, \infty ) \bigr) \evec{f} \| ( \| X \| + 1 );
\end{align*}
since $t \mapsto \sigma_t( X )$ is strongly continuous on $\fock$ for
all $X \in \bop{\fock}{}$, this estimate gives the
existence of $\sigma_S( X ) \evec{f}$: note that
$s \mapsto \evecc\bigl( f( \cdot + s ) \bigr)$ is uniformly continuous
on $[ 0, T ]$ and
\begin{multline*}
\sup\{ \| ( \sigma_r( X ) - X ) %
\evecc\bigl( f( \cdot + s ) \bigr) \| : %
r \in [ 0, \delta ], \ s \in [ 0, T ] \} \\[1ex]
\le \sup\{ \| ( \sigma_r( X ) - X ) %
\evecc\bigl( f( \cdot + s_j ) \bigr) \| : %
r \in [ 0, \delta ], \ j = 0, \ldots, N \} \quad \\
+ 2 \| X \| \sup\{ \| \evecc\bigl( f( \cdot + s ) \bigr) - %
\evecc\bigl( f( \cdot + s_j ) \bigr) \| : %
s \in [ s_j, s_{j + 1} ], \ j = 0, \ldots, N \},
\end{multline*}
where $\{ 0 = s_0 < \cdots < s_{N + 1} = T \}$ is any partition of
$[ 0, T ]$. Thus $\sigma_S( X ) x := \lim_\pi \sigma_{S, \pi}( X ) x$
exists if $x$ is any finite linear combination of exponential vectors
corresponding to functions with compact support; since
$\| \sigma_{S, \pi}( X ) \| \le \| X \|$ for all $\pi$, it follows
that $x \mapsto \sigma_S( X ) x$ extends to a bounded operator
$\sigma_S( X )$ on the whole of $\fock$; a simple approximation
argument now shows that~$\sigma_S( X )$ is the limit of
$\sigma_{S, \pi}( X )$ in the strong operator topology.

The map $\sigma_S$ is $*$-homomorphic and unital because
each~$\sigma_{S, \pi}$ is. For the final identity, observe that
\[
\sigma_{S, \pi}( X ) %
\sum_{k = 1}^{n + 1} S\bigl( ( \pi_{k - 1}, \pi_k ] \bigr) %
\Gamma_{\pi_k} = %
\sum_{j = 1}^{n + 1} S\bigl( ( \pi_{j - 1}, \pi_j ] \bigr) %
\sigma_{\pi_j}( X ) \Gamma_{\pi_j} = %
\sum_{j = 1}^{n + 1} S\bigl( ( \pi_{j - 1}, \pi_j ] \bigr) %
\Gamma_{\pi_j} X,
\]
by \eqref{eqn:shift}, so refining $\pi$ gives the claim.
\end{proof}

\begin{proposition}\label{prp:sigmagamma}
Let $S$ be a finite quantum stop time. Then
\[
\sigma_S( X ) = j_S( I_{S)} \otimes \Gamma_S X \Gamma_S^* ) j_S^* %
\qquad \text{for all } X \in \bop{\fock}{},
\]
where $j_S : \fock_{S)} \otimes \fock_{[S} \to \fock$ is the isometric
isomorphism of Theorem~\ref{thm:isom}.
\end{proposition}
\begin{proof}
Let $\pi = \{ 0 = \pi_0 < \cdots < \pi_{n + 1} = \infty \}$ be a
finite partition containing $t = \pi_m \in ( 0, \infty )$ and let $f$,
$f'$, $g$, $g' \in \elltwo$. If $r$, $s \in \R_+$ are such that
$r < s$ then
\[
\langle \evec{1_{[ 0, s )} f + \theta_s g}, %
S\bigl( ( r, s ] \bigr) \sigma_s( X ) %
\evec{1_{[ 0, s )} f' + \theta_s g'} \rangle = %
\langle \evec{1_{[ 0, s )} f}, %
S\bigl( ( r, s ] \bigr) \evec{1_{[ 0, s )} f'} \rangle %
\langle \evec{g}, X \evec{g'} \rangle.
\]
With the notation used in the proofs of Lemma~\ref{lem:key} and
Theorems~\ref{thm:expS} and~\ref{thm:shift}, it follows that
\[
\langle I_{S, \pi \cap [ 0, t ]}\bigl( f, \Gamma \evec{g} \bigr), %
\sigma_{S, \pi}( X ) %
I_{S, \pi \cap [ 0, t ]}\bigl( f', \Gamma\evec{g'} \bigr) %
\rangle = \langle E_{S, \pi \cap [ 0, t ]} \evec{f}, %
E_{S, \pi \cap [ 0, t ]} \evec{f'} \rangle %
\langle \evec{g}, X \evec{g'} \rangle
\]
and therefore, by refining $\pi$,
\begin{multline*}
\langle %
j_S\bigl( E_{S, t} \evec{f} \otimes \Gamma_S \evec{g} \bigr), %
\sigma_S( X ) j_S\bigl( E_{S, t} \evec{f'} \otimes %
\Gamma_S \evec{g'} \bigr) \rangle \\
 = \langle E_{S, t} \evec{f} \otimes \Gamma_S \evec{g}, %
( I_{S)} \otimes \Gamma_S X \Gamma_S^* ) %
\bigl( E_{S, t} \evec{f'} \otimes \Gamma_S \evec{g'} \bigr) \rangle,
\end{multline*}
as required.
\end{proof}

\begin{theorem}\label{thm:addt}
The map $S \mapsto \sigma_S$ is a homomorphism from the convolution
semigroup of finite quantum stop times to the semigroup of unital
endomorphisms of $\bop{\fock}{}$ : for any finite quantum stop times
$S$ and $T$,
\[
\sigma_{S \star T} = \sigma_S \comp \sigma_T
\]
and, furthermore, $\sigma_S( E_T ) = E_{S \star T}$.
\end{theorem}
\begin{proof}
Let $X \in \bop{\fock}{}$. Proposition~\ref{prp:sigmagamma} and
Theorem~\ref{thm:SstarTiso} imply that
\begin{align*}
\sigma_S\bigl( \sigma_T( X ) \bigr) = %
j_S( I_{S)} \otimes \Gamma_S j_T( I_{T)} \otimes %
\Gamma_T X \Gamma_T^* ) j_T^* \Gamma_S^* ) j_S^* 
 & = j_{S, T}( I_{S)} \otimes I_{\Gamma_S( \fock_{T)} )} \otimes %
\Gamma_{S \star T} X \Gamma_{S \star T}^* ) j_{S, T}^* \\[1ex]
 & = j_{S \star T}( I_{S \star T)} \otimes %
\Gamma_{S \star T} X \Gamma_{S \star T}^* ) j_{S \star T}^* \\[1ex]
 & = \sigma_{S \star T}( X ).
\end{align*}
The second identity follows from this; with the notation used in the
proof of Theorem~\ref{thm:expS},
\[
\sigma_S( E_0 ) = \stlim_\pi \sigma_{S, \pi}( E_0 ) = %
\stlim_\pi E_{S, \pi} = E_S
\]
and therefore
\[
\sigma_S( E_T ) = \sigma_S\bigl( \sigma_T( E_0 ) \bigr) = %
\sigma_{S \star T}( E_0 ) = E_{S \star T}.
\qedhere
\]
\end{proof}

\section{Isometric cocycles}

\begin{notation}
Let $\vna \subseteq \bop{\ini}{}$ be a von~Neumann algebra, where
$\ini$ is a complex Hilbert space, and let $\uwkten$ denote the
ultraweak tensor product. Henceforth, $S$ is a finite quantum stop
time extended by ampliation to act on $\ini \otimes \fock$; the time
projection~$E_S$ extends by ampliation in the same manner.
\end{notation}

\begin{definition}
Let $p$ be an orthogonal projection in $\bop{\mul}{}$ which acts
pointwise on~$L^2\bigl( [ t, \infty ) ; \mul \bigr)$ for all
$t \in \R_+$, so that
\[
( p f )( s ) := p f( s ) \qquad \text{for all } s \in [ t, \infty ) %
\text{ and } f \in L^2\bigl( [ t, \infty ); \mul \bigr),
\]
and let $P_{[t} \in \bop{\fock_{[t}}{}$ be the second quantisation
of~$p$ acting in this way, so that $P_{[t} \evec{f} = \evec{p f}$ for
all $f \in L^2\bigl( [ t, \infty ); \mul \bigr)$.

A family of operators
$V = ( V_t )_{t \in \R_+} \subseteq \vna \uwkten \bop{\fock}{}$ is a
\textit{$p$-adapted process} if and only if
\begin{equation}\label{eqn:padapt}
V_t = V_{t)} \otimes P_{[t} \qquad \text{for all } t \in \R_+,
\end{equation}
where $V_{t)} \in \vna \uwkten \bop{\fock_{t)}}{}$. If $p = I_\mul$
then this is identity adaptedness, the usual choice for
Hudson--Parthasarathy quantum stochastic calculus; if $p = 0$ then
this is vacuum adaptedness \cite{Blt01}. Note that $V$ is vacuum
adapted if and only if $V_t = E_t V_t E_t$ for all $t \in \R_+$.

Given a $p$-adapted process $V$, its
\emph{identity-adapted projection} is the process
$\wh{V} = \bigl( \wh{V}_t \bigr)_{t \in \R_+}$, where
\[
\wh{V}_t := V_{t)} \otimes I_{[t} \qquad \text{for all } t \in \R_+.
\]

A $p$-adapted process $V$ is a \emph{cocycle} (more fully, a
\emph{left operator Markovian cocycle}) if and only if
\begin{equation}\label{eqn:defcocycle}
V_{s + t} = \wh{V}_{s} \, \sigma_s( V_t ) %
\qquad \mbox{for all } s, t \in \R_+,
\end{equation}
where $\sigma$ is the CCR flow of Definition~\ref{def:ccrflow},
extended by ampliation.

A $p$-adapted cocycle $V$ is \textit{isometric} if
$V_{t)}^* V_{t)} = I_\ini \otimes I_{t)}$ for all $t \in \R_+$, where
$V_{t)}$ is as in~\eqref{eqn:padapt}; equivalently,
$\wh{V}_t^* \wh{V}_t = I_\ini \otimes I$ for all $t \in \R_+$. For
such an isometric cocycle $V$, it holds that~$\| V_t \| = 1$ and
\[
V_t^* V_t = I_\ini \otimes I_{t)} \otimes P_{[t} \in %
I_\ini \otimes I_{t)} \otimes \bop{\fock_{[t}}{} %
\qquad \text{for all } t \in \R_+.
\]
\end{definition}

\begin{example}\label{eg:weylcocycle}
Let $W( f ) \in \bop{\fock}{}$ be the \emph{Weyl operator}
corresponding to $f \in \elltwo$, so that
\[
W( f ) \evec{g} = %
\exp\bigl( -\hlf \| f \|^2 - \langle f, g \rangle \bigr) %
\evec{f + g} \qquad \text{and} \qquad W( f ) W( g ) = %
\exp\bigl( -\I \langle f, g \rangle \bigr) \evec{f + g}
\]
for all $f$, $g \in \elltwo$. Then
$\bigl( W_t( c ) := %
I_\ini \otimes W( 1_{[ 0, t )} c ) \bigr)_{t \ge 0}$ and
$\bigl( V_t( c ) := %
I_\ini \otimes E_t W( 1_{[ 0, t )} c ) \bigr)_{t \ge 0}$ are
isometric cocycles for all $c \in \mul$, which are $1$-adapted and
$0$-adapted respectively.

More examples of isometric cocycles may be constructed as the
solutions of quantum stochastic differential equations: see
\cite[Section~5]{Lin05} and references therein.
\end{example}

\begin{lemma}\label{lem:uni}
Let $V$ be an isometric $p$-adapted cocycle and let
$t \in ( 0, \infty )$. For any finite partition
$\pi = \{ 0 = \pi_0 < \cdots < \pi_{n + 1} = t \}$, the Riemann sum
\[
V_{S, \pi} := %
\sum_{j = 1}^{n + 1} V_{\pi_j} S\bigl( ( \pi_{j - 1}, \pi_j ] \bigr)
\]
is such that
\[
\| V_{S, \pi} z \| \le \| S\bigl( [ 0, t ] \bigr) z \| %
\qquad \text{for all } z \in \ini \otimes \fock.
\]
\end{lemma}
\begin{proof}
If $r$, $s \in \R_+$ are such that $r \le s$ then, by the cocycle,
$p$-adaptedness and isometry properties,
\begin{equation}\label{eqn:cocycle1}
V_s^* V_r = \sigma_r( V_{s - r}^* ) \, ( V_{r)}^* \otimes I_{[r} ) %
( V_{r)} \otimes P_{[r} ) \in %
\vna \uwkten I_{r)} \uwkten \bop{\fock_{[r}}{}
\end{equation}
and
\begin{equation}\label{eqn:cocycle2}
V_r^* V_s = ( V_{r)}^* \otimes P_{[r} )( V_{r)}^* \otimes I_{[r} )
\, \sigma_r( V_{s - r} ) %
\in \vna \uwkten I_{r)} \uwkten \bop{\fock_{[r}}{}.
\end{equation}
Consequently, if $j$, $k \in \{ 1, \ldots, n + 1 \}$ are such that
$j \neq k$ then, as $S$ is identity adapted,
\[
S\bigl( ( \pi_j, \pi_{j + 1} ] \bigr) V_{\pi_{j + 1}}^* %
V_{\pi_{k + 1}} S\bigl( ( \pi_k, \pi_{k + 1} ] \bigr) = 0,
\]
and thus, for all $z \in \ini \otimes \fock$, we have that
\[
\| V_{S, \pi} z \|^2 = %
\sum_{j = 1}^{n + 1} \| %
( I_\ini \otimes I_{\pi_j)} \otimes P_{[\pi_j} ) %
S\bigl( ( \pi_{j - 1}, \pi_j ] \bigr) z \|^2 \le %
\sum_{j = 1}^{n + 1} %
\| S\bigl( ( \pi_{j - 1}, \pi_j ] \bigr) z \|^2 = %
\| S\bigl( [ 0, t ] \bigr) z \|^2.
\qedhere
\]
\end{proof}

\begin{theorem}\label{thm:stopcocycle}
Let $V$ be an isometric $p$-adapted cocycle and let
$t \in ( 0, \infty )$. If $V$ is strongly continuous then
\[
V_{S, t} := \stlim_\pi V_{S, \pi}
\]
is well defined on $\ini \otimes \fock$. Furthermore,
$V_{S, t} = V_{S, t} S\bigl( [ 0, t ] \bigr)$ and
\[
\| V_{S, t} z \| \le \| S\bigl( [ 0, t ] \bigr) z \| %
\qquad \mbox{for all } z \in \ini \otimes \fock.
\]
\end{theorem}
\begin{proof}
Let $\pi'$ be a refinement of
$\pi = \{ 0 = \pi_0 < \cdots < \pi_{n + 1} = t \}$ and, for $j = 0$,
\ldots, $m$, let~$k_j \ge j$ be such that $\pi'_{k_j} = \pi_j$ and
let~$l_j \ge 1$ be such that $\pi'_{k_j + l_j} = \pi_{j + 1}$. Then
\[
V_{S, \pi'} - V_{S, \pi} = %
\sum_{j = 0}^n \sum_{l = 1}^{l_j} %
( V_{\pi'_{k_j + l}} - V_{\pi'_{k_j + l_j}} ) %
S\bigl( ( \pi'_{k_j + l - 1}, \pi'_{k_j + l} ] \bigr).
\]
If $a$, $b$, $c$, $p$, $q$, $r \in \R_+$ are such that $a < b \le c$
and $p < q \le r$ then, by \eqref{eqn:cocycle1} and
\eqref{eqn:cocycle2},
\begin{align*}
( V_c - V_b )^* ( V_r - V_q ) = %
V_c^* V_r - V_c^* V_q - V_b^* V_r + V_b^* V_q \in %
\bop{\ini}{} \uwkten I_{b \wedge q)} \uwkten %
\bop{\fock_{[b \wedge q}}{}
\end{align*}
and therefore
\[
S\bigl( ( a, b ] \bigr) ( V_c - V_b )^* %
( V_r - V_q ) S\bigl( ( p, q ] \bigr) = \left\{ %
\begin{array}{ll}
( V_c - V_b )^* ( V_r - V_q ) S\bigl( ( a, b ] \bigr) %
S\bigl( ( p, q ] \bigr) & %
\mbox{if } b \le q, \\[1ex]
S\bigl( ( a, b ] \bigr) S\bigl( ( p, q ] \bigr) %
( V_c - V_b )^* ( V_r - V_q ) & \mbox{if } q \le b.
\end{array}\right.
\]
Furthermore, if $Z = X \otimes I_{q)} \otimes Y$, where $X \in \vna$
and $Y \in \bop{\fock_{[q}}{}$, then
\begin{align*}
\langle u\evec{f}, Z S\bigl( ( p, q ] \bigr) %
u\evec{f} \rangle & = \langle u\evec{1_{[ q, \infty )} f}, %
( X \otimes Y ) u\evec{1_{[ q, \infty )} f} \rangle \, %
\langle \evec{1_{[ 0, q)} f}, S\bigl( ( p, q ] \bigr) %
\evec{1_{[ 0, q )} f} \rangle \\[1ex]
& = \langle u\evec{f}, Z u\evec{f} \rangle \, %
\langle \evec{f}, S\bigl( ( p, q ] \bigr) \evec{f} \rangle \, %
\| \evec{f} \|^{-2}
\end{align*}
and therefore
\[
\| ( V_r - V_q ) S\bigl( ( p, q ] \bigr) u\evec{f} \|^2 \le %
\| ( V_r - V_q ) u\evec{f} \|^2 %
\| S\bigl( ( p, q ] \bigr) \evec{f} \|^2.
\]
Hence
\begin{align*}
\| ( V_{S, \pi'} - V_{S, \pi} ) u \evec{f} \|^2 & = %
\sum_{j = 0}^n \sum_{l = 1}^{l_j} %
\bigl\| ( V_{\pi'_{k_j + l}} - V_{\pi'_{k_j + l_j}} ) %
S\bigl( ( \pi'_{k_j + l - 1}, \pi'_{k_j + l} ] \bigr) u\evec{f} %
\bigr\|^2 \\[1ex]
 & \le \sum_{j = 0}^n \sum_{l = 1}^{l_j} %
\| ( V_{\pi'_{k_j + l}} - V_{\pi'_{k_j + l_j}} ) u\evec{f} \|^2 %
\| S\bigl( ( \pi'_{k_j + l - 1}, \pi'_{k_j + l} ] \bigr) \evec{f} \|^2
\\[1ex]
 & \le %
\sup\{ \| ( V_r - V_{\pi_j} ) u\evec{f} \|^2 : %
r \in [ \pi_j, \pi_{j + 1} ], \ j = 0, \ldots, n \} \, %
\| S\bigl( [ 0, t ] \bigr) \evec{f} \|^2.
\end{align*}
This gives the first claim on the algebraic tensor product
$\ini \algten \evecs$. It follows from Lemma~\ref{lem:uni}
that~$V_{S, t}$ is a contraction on $\ini \algten \evecs$, so it
extends to a bounded operator on the whole of~$\ini \otimes \fock$,
and an approximation argument now gives strong convergence everywhere.

The second claim holds because
$V_{S, \pi} S\bigl( [ 0, t ] \bigr) = V_{S, \pi}$; the third is an
immediate consequence of Lemma~\ref{lem:uni}.
\end{proof}

\begin{corollary}
If the isometric $p$-adapted cocycle $V$ is strongly continuous then
\[
V_S = V_{S, \infty} := \stlim_{t \to \infty} V_{S, t}
\]
is well defined.
\end{corollary}
\begin{proof}
Suppose $s$, $t \in \R_+$ are such that $s < t$ and let
$\pi = \{ 0 = t_0 < \cdots < t_{n + 1} = t \}$ be a partition of
$[ 0, t ]$ with $\pi_m = s$. If $z \in \ini \otimes \fock$ then
\[
\| ( V_{S, \pi} - V_{S, \pi'} ) z \|^2 = \sum_{j = m + 1}^{n + 1} \| %
V_{\pi_j} S\bigl( ( \pi_{j - 1}, \pi_j ] \bigr) z \|^2 \le %
\| S\bigl( ( s, t ] \bigr) z \|^2.
\]
Hence
$\| ( V_{S, t} - V_{S, s} ) z \| \le \| S\bigl( ( s, t ] \bigr) z \|$,
which establishes the Cauchy nature of the net
$( V_{S, t} z )_{t \in \R_+}$ for every $z \in \ini \otimes \fock$,
and from this follows the existence of the limit.
\end{proof}

\begin{proposition}\label{prp:vnorm}
Let $V$ be an isometric vacuum-adapted cocycle which is strongly
continuous and let $t \in ( 0, \infty ]$. Then
\[
V_{S, t} E_{S, t} = V_{S, t} \qquad \text{and} \qquad %
\| V_{S, t} z \| = \| E_{S, t} z \| %
\qquad \text{for all } z \in \ini \otimes \fock.
\]
\end{proposition}
\begin{proof}
Suppose first that $t \in ( 0, \infty )$ and let
$\pi = \{ 0 = \pi_0 < \cdots < \pi_{n + 1} = t\}$. Then, with the
notation used in the proof of Theorem~\ref{thm:expS},
\[
V_{S, \pi} E_{S, \pi} = %
\sum_{j = 1}^{n + 1} V_{\pi_j} S\bigl( ( \pi_{j - 1}, \pi_j ] \bigr) %
\sum_{k = 1}^{n + 1} S\bigl( ( \pi_{k - 1}, \pi_k ] \bigr) E_{\pi_k} %
= V_{S, \pi},
\]
since $V$ is vacuum adapted and $E_{\pi_j}$ commutes with
$S\bigl( ( \pi_{j - 1}, \pi_j ] \bigr)$ for all $j$. Furthermore, the
working in the proof of Lemma~\ref{lem:uni} shows that
\[
\| V_{S, \pi} z \|^2 = %
\sum_{j = 1}^{n + 1} \| S\bigl( ( \pi_{j - 1}, \pi_j ] \bigr) %
E_{\pi_j} z \|^2 = \| E_{S, \pi} z \|^2 %
\qquad \text{for all } z \in \ini \otimes \fock.
\]
Refining $\pi$ now gives both identities as claimed; the remaining
case follows by letting $t \to \infty$.
\end{proof}

\begin{proposition}\label{prp:inorm}
Let $V$ be an isometric identity-adapted cocycle which is strongly
continuous and let $t \in ( 0, \infty ]$. Then
\[
\| V_{S, t} z \| = \| S\bigl( [ 0, t ] \bigr) z \| %
\qquad \text{for all } z \in \ini \otimes \fock.
\]
\end{proposition}
\begin{proof}
If $\pi$ is a finite partition of $[ 0, t ]$, where
$t \in ( 0, \infty )$, and $z \in \ini \otimes \fock$ then the proof
of Lemma~\ref{lem:uni} gives that
$\| V_{S, \pi} z \| = \| S\bigl( [ 0, t ] \bigr) z \|$.
The result follows by refining $\pi$ and then letting~$t \to \infty$.
\end{proof}

\begin{remark}
Propositions~\ref{prp:vnorm} and~\ref{prp:inorm} imply that $V_{S, t}$
is a partial isometry for all $t \in [ 0, \infty ]$ if~$V$ is vacuum
adapted or identity adapted, and $V_S$ is an isometry in the latter
case.
\end{remark}

\section{A stopped cocycle relation}\label{sec:relation}

\begin{notation}
As in the preceding section, $S$ is a finite quantum stop time
extended by ampliation to act on $\ini \otimes \fock$; the maps $E_S$,
$\Gamma_S$ and $\sigma_S$ extend similarly.
\end{notation}

\begin{theorem}\label{thm:cocyclerel}
Let $S$ be a finite quantum stop time and let $V$ be an isometric
$p$-adapted cocycle. Then
\begin{equation}\label{eqn:stoppedcocycle}
V_{S + t} = \wh{V}_S \, \sigma_S( V_t ) %
\qquad \text{for all } t \in \R_+.
\end{equation}
\end{theorem}
\begin{proof}
Let $t$, $T \in \R_+$ be such that $t < T$, let
$\pi = \{ 0 = \pi_0 < \cdots < \pi_{p + 1} = \infty \}$ be a finite
partition which contains $t = \pi_m$ and $T = \pi_{n + 1 }$, and
let $\pi' := \pi - t$ be the partition of $[ 0, \infty ]$ such
that~$\pi'_j = \pi_{j + m} - t$ for $j = 0$, \ldots, $p + 1 - m$.
Then, in the notation of Lemma~\ref{lem:uni} and
Theorem~\ref{thm:flowstop},
\begin{align*}
V_{S + t, \pi \cap [ 0, T ]} = \sum_{j = 1}^{n + 1} V_{\pi_j} \, %
( S + t )\bigl( ( \pi_{j - 1}, \pi_j ] \bigr) & = %
\sum_{j = m + 1}^{n + 1} \wh{V}_{\pi_j - t} \, %
\sigma_{\pi_j - t}( V_t ) \, %
S\bigl( ( \pi_{j - 1} - t, \pi_j - t ] \bigr) \\[1ex]
 & = \sum_{j = 1}^{n + 1 - m} \wh{V}_{\pi'_j} %
S\bigl( ( \pi'_{j - 1}, \pi'_j ] \bigr) %
\sum_{k = 1}^{p + 1 - m} S\bigl( ( \pi'_{k - 1}, \pi'_k ] \bigr) %
\, \sigma_{\pi'_k}( V_t ) \\[1ex]
 & = \wh{V}_{S, \pi' \cap [ 0, T - t ]} \, \sigma_{S, \pi'}( V_t ).
\end{align*}
Refining $\pi$ and then letting $T \to \infty$ gives the result.
\end{proof}

\begin{remark}
The identity \eqref{eqn:stoppedcocycle} is a non-deterministic version
of \eqref{eqn:defcocycle}, the defining property of a left cocycle.
\end{remark}

\begin{remark}
In \cite{App88}, Applebaum considers stopping a unitary
identity-adapted process $U$ which satisfies the cocycle identity
\begin{equation}\label{eqn:appdet}
U_t = \Gamma_s^* U_s^* U_{s + t} \Gamma_s %
\qquad \text{for all } s, t \in \R_+
\end{equation}
and the localisation property
\begin{equation}\label{eqn:apploc}
U_s^* U_{s + t} \in \bop{\ini}{} \uwkten I_{s)} \uwkten %
\bop{\fock_{[ s, s + t )}}{} \uwkten I_{[s + t} %
\qquad \text{for all } s, t \in ( 0, \infty );
\end{equation}
for such processes, these conditions are equivalent to being a left
operator Markovian cocycle. The identity
\begin{equation}\label{eqn:appstop}
U_{S + t} = U_S \Gamma_S U_t \Gamma_S^* %
\qquad \text{for all } t \in \R_+
\end{equation}
is obtained \cite[Theorem~4.3]{App88}, where $S$ is any finite quantum
stop time. However, $\Gamma_S U_t \Gamma_S^*$ is taken to act on the
range of the isometry $\Gamma_S$ \cite[(4.1)]{App88}, so it is clearer
to write \eqref{eqn:appstop} in the following manner:
\[
U_{S + t} \Gamma_S = U_S \Gamma_S U_t \Gamma_S^* \Gamma_S = %
U_S \Gamma_S U_t.
\]
From Theorem~\ref{thm:cocyclerel}, if $S$ is any finite quantum stop
time and $V$ is any isometric $p$-adapted cocycle then, by
\eqref{eqn:flowgamma},
\[
V_{S + t} = \wh{V}_S \, \sigma_S( V_t ) \ \Longrightarrow \ %
V_{S + t} \Gamma_S = \wh{V}_S \, \sigma_S( V_t ) \, \Gamma_S = %
\wh{V}_S \Gamma_S V_t,
\]
which is the identity obtained by Applebaum.
Furthermore, Proposition~\ref{prp:inorm} gives that
\[
\wh{V}_S^* V_{S + t} = \wh{V}_S^* \wh{V}_S \, \sigma_S( V_t ) = %
\sigma_S( V_t ) \in \im \sigma_S,
\]
which generalises the localisation condition
\[
U_s^* U_{s + t} \in %
\bop{\ini}{} \uwkten I_{s)} \uwkten \bop{\fock_{[s}}{} = %
\im \sigma_s,
\]
and
\[
\Gamma_S^* \wh{V}_S^* V_{S + t} \Gamma_S = %
\Gamma_S^* \sigma_S( V_t ) \Gamma_S = %
\Gamma_S^* \Gamma_S V_t = V_t,
\]
which is the stopped version of \eqref{eqn:appdet}.

If $V$ is vacuum adapted and $t \in \R_+$ then, by
Theorem~\ref{thm:addt},
\[
\wh{V}_S^* V_{S + t} = \sigma_S( V_t ) = \sigma_S( E_t V_t E_t ) = %
\sigma_S( E_t ) \, \sigma_S( V_t ) \, \sigma_S( E_t ) = %
E_{S + t} \wh{V}_S^* V_{S + t} E_{S + t},
\]
so $\wh{V}_S^* V_{S + t}$ is vacuum adapted at $S + t$. Similarly, if
$V$ is identity adapted and $t \in \R_+$ then Theorem~\ref{thm:addt}
gives that
\[
\sigma_S( V_t ) \, \sigma_{S + t}( X ) = %
\sigma_S\bigl( V_t \, \sigma_t( X ) \bigr) = %
\sigma_S\bigl( \sigma_t( X ) \, V_t \bigr) = %
\sigma_{S + t}( X ) \, \sigma_S( V_t )
\]
for all $X \in \vna' \uwkten \bop{\fock}{}$, so
\[
\wh{V}_S^* V_{S + t} = \sigma_S( V_t ) \in %
\sigma_{S + t}\bigl( \vna' \uwkten \bop{\fock}{} \bigr)' = %
\bigl( \vna' \uwkten I_{S + t)} \uwkten %
\bop{\fock_{[S + t}}{} \bigr)' = %
\vna \uwkten \bop{\fock_{S + t)}}{} \uwkten I_{[S + t}
\]
and $\wh{V}_S^* V_{S + t}$ is identity adapted at $S + t$.
\end{remark}

\subsection*{Acknowledgements}
The first author is grateful to the Jawaharlal Nehru Centre for
Advanced Scientific Research, Bangalore, for its hospitality, and
to the laboratoire de math\'{e}matiques de Besan\c{c}on,
Universit\'{e} de Franche-Comt\'{e}, where helpful conversations with
Professor Uwe Franz took place. Both authors acknowledge support from
the UKIERI research network \emph{Quantum Probability, Noncommutative
Geometry and Quantum Information}.


\section*{References}

\begin{thebibliography}{99}

\bibitem{App88}
\textsc{D.~Applebaum},
Stopping unitary processes in Fock space,
\textit{Publ.\ RIMS Kyoto Univ.}~24 (1988), 697--705.

\bibitem{Arv03}
\textsc{W.~Arveson},
\textit{Noncommutative Dynamics and $E$-Semigroups},
Monogr.\ Math., Springer-Verlag, New York, 2003.

\bibitem{AtS98}
\textsc{S.~Attal} \& \textsc{K.~B.~Sinha},
Stopping semimartingales on Fock space,
in
\textit{Quantum Probability Communications~X}
(R.~L.~Hudson \& J.~M.~Lindsay, eds.), 171--185,
World Scientific, Singapore, 1998.

\bibitem{BaL86}
\textsc{C.~Barnett} \& \textsc{T.~Lyons},
Stopping noncommutative processes,
\textit{Math.\ Proc.\ Cambridge Philos.\ Soc.}~99 (1986), no.~1,
151--161.

\bibitem{BaT90}
\textsc{C.~Barnett} \& \textsc{B.~Thakrar},
A noncommutative random stopping theorem,
\textit{J.\ Funct.\ Anal.}~88 (1990), no.~2, 342--350.

\bibitem{Blt01}
\textsc{A.~C.~R.~Belton},
Quantum $\Omega$-semimartingales and stochastic evolutions,
\textit{J.\ Funct.\ Anal.}~187 (2001), no.~1, 94--109.

\bibitem{Coq06}
\textsc{A.~Coquio},
The optional stopping theorem for quantum martingales,
\textit{J.\ Funct.\ Anal.}~238 (2006), no.~1, 149--180.

\bibitem{Fox76}
\textsc{D.~W.~Fox},
Spectral measures and separation of variables,
\textit{J.\ Res.\ Nat.\ Bur.\ Standards Sect.~B}~80B (1976), no.~3,
347--351.

\bibitem{Hud79}
\textsc{R.~L.~Hudson},
The strong Markov property for canonical Wiener processes,
\textit{J.\ Funct.\ Anal.}~34 (1979), no.~2, 266--281.

\bibitem{Hud07}
\textsc{R.~L.~Hudson},
Stop times in Fock space quantum probability,
\textit{Stochastics}~79 (2007), no.~3--4, 383--391.

\bibitem{Lin05}
\textsc{J.~M.~Lindsay},
Quantum stochastic analysis -- an introduction,
in
\textit{Quantum Independent Increment Processes~I}
(M.~Sch\"{u}rmann and U.~Franz, eds.), 181--271,
Lecture Notes in Math.~1865, Springer-Verlag, Berlin, 2005.

\bibitem{Par92}
\textsc{K.~R.~Parthasarathy},
\textit{An Introduction to Quantum Stochastic Calculus},
Monographs in Mathematics~85, Birkh\"auser Verlag, Basel, 1992.

\bibitem{PaS87}
\textsc{K.~R.~Parthasarathy} \& \textsc{K.~B.~Sinha},
Stop times in Fock space stochastic calculus,
\textit{Probab.\ Theory Related Fields} 75 (1987), no.~3, 317--349.

\bibitem{Sau88}
\textsc{J.-L.~Sauvageot},
First exit time: a theory of stopping times in quantum processes,
in
\textit{Quantum Probability and Applications~III}
(L.~Accardi \& W.~von~Waldenfels, eds.), 285--299,
Lecture Notes in Math.~1303, Springer-Verlag, Berlin, 1988.

\end{thebibliography}
\end{document}